\documentclass[12pt,a4]{article}
\usepackage[utf8]{inputenc}
\usepackage[english]{babel}
\usepackage{amssymb}
\usepackage{amsmath}
\usepackage{amsthm}
\usepackage{setspace}
\usepackage{graphicx}
\usepackage{hyperref}

\makeatletter
\newcommand\xleftrightarrow[2][]{%
  \ext@arrow 9999{\longleftrightarrowfill@}{#1}{#2}}
\newcommand\longleftrightarrowfill@{%
  \arrowfill@\leftarrow\relbar\rightarrow}
\makeatother

\hypersetup{
    colorlinks=true,
    linkcolor=blue,
    filecolor=magenta,      
    urlcolor=cyan,
}

\newtheorem{proposition}{Proposition}
\newtheorem{lemma}[proposition]{Lemma}
\newtheorem{theorem}[proposition]{Theorem}
\newtheorem{corollary}[proposition]{Corollary}
\theoremstyle{definition}

\newtheorem{definition-theorem}[proposition]{Definition--Theorem}
\newtheorem{definition-proposition}[proposition]{Definition--Proposition}

\numberwithin{equation}{section}
\numberwithin{proposition}{section}

\usepackage[a4paper]{geometry}

\title{Commutation of transfer and Aubert-Zelevinski involution for metaplectic groups}
\author{Fei Chen}
\date{}

\begin{document}
\maketitle
\begin{abstract}
    A result of K. Hiraga says endoscopic transfer is compatible with Aubert-Zelevinski involution. In this short note, we generalize Hiraga's result to metaplectic group setting.
\end{abstract}
\setlength{\unitlength}{2.5cm}
\section{Introduction}
\paragraph{} The theory of endoscopy was developed to understand the internal structure of L-packets and it is foundational to the stabilization of the trace
formula. The formulation of endoscopy was generalized by W-W. Li \cite{li2011transfert} to metaplectic groups and had lead to the recent stabilization of trace formula for metaplectic groups \cite{li2021stabilization}.

An important property of endoscopic transfer, proved by Hiraga \cite{Hiraga} 
 for ordinary endoscopy and by B. Xu \cite[Appendix]{xu2017moeglin} for twisted endoscopy, is that Aubert-Zelevinski  involution is compatible with endoscopic transfer. This property plays a role in Arthur's endoscopic classification \cite{arthur2013endoscopic} when one want to study non-tempered representations.

The goal of this short note is to prove endoscopic transfer for metaplectic groups  is also compatible with Aubert-Zelevinski involution:

\begin{theorem}\label{main}
    
Let $\tilde{G}$ be a metaplectic group and let $G^!$ be an  endoscopic datum of $\tilde{G}$. Then for any stable virtual character $\Theta$ of $G^!$ we have

\begin{equation*}
    {D}_{\Tilde{G}}\circ \mathcal{T}^{\vee}_{G^!,\tilde{G}}(\Theta)=\mathcal{T}^{\vee}_{G^!,\tilde{G}}\circ D_{G^!}(\Theta)
\end{equation*}
where $D$ stands for  the Aubert-Zelevinski involution and $\mathcal{T}^{\vee}_{G^!,\tilde{G}}$ stands for endoscopic transfer for distributions.
\end{theorem}

The proof is based on the idea of Hiraga. However, there are some new features in metaplectic group case which makes it (at least to the author) not clear whether Hiraga's result can be applied word for word. For example, when we considering relation of $\mathcal{T}^{\vee}_{G^!,\tilde{G}}$ with parabolic induction and Jacquet module, we need to include certain twists.  Besides, some steps of Hiraga's proof uses combinatorial properties of Weyl groups, these are replaced by arguments using explicit partitions in this note.
\paragraph{Organization}The structure of this article is as follows. In Section \ref{preliminary}, we  recall facts about endoscopy for $\Tilde{G}$ established by Li. In Section \ref{3}, we discuss  compatibility of endoscopic transfer with Jacquet functors. In the last section , we prove Theorem \ref{main}.
\section{Preliminaries}\label{preliminary}
\paragraph{}In this section, we fix our notations and recall facts about metaplectic groups and the theory of endoscopy for them following \cite{li2011transfert,li2019spectral,li2021stabilization}. We will work over a non-Archimedean local field $F$ such that $\text{char}(F) =0 $. All algebraic groups 
 $H$ are defined over $F$ and we use the same notation for their group of $F$-points.
\subsection{Metaplectic groups}

\paragraph{}Let $\operatorname{Sp}(2n)$ be the symplectic group associated to a $2n$-dimensional symplectic vector space. After fixing a Borel subgroup and a maximal torus inside, all the standard Levi subgroups of $\operatorname{Sp}(2n)$ takes the form 
$M=\prod_{i=1}^k\operatorname{GL}(n_i) \times \operatorname{Sp}(2m)$ where $n_i\in \mathbb{Z}_{>0}$, $m\in \mathbb{Z}_{\geq 0}$  and  $\sum_{i=1}^kn_i+m=n$. The set of standard Levi subgroup is in bijection with $(\{n_i\}_i,m)$ of the above form.

Let $\operatorname{SO}(2n+1)$ be the split odd orthogonal group over $F$. Then after fixing a Borel pair, all standard Levi subgroup takes the form $\prod_{i=1}^k\operatorname{GL}(n_i) \times \operatorname{SO}(2m+1)$.

We fix a non-trivial additive character $\psi:F \to \mathbb{C}^{\times}$, then the Weil representation give rise to a central extension 
$$1\to \mathbb{C}^{\times} \to \overline{\operatorname{Sp}}_{\psi}(2n) \to \operatorname{Sp}(2n) \to 1 .$$
It is well known that the derived subgroup of $\overline{\operatorname{Sp}}_{\psi}(2n)$ is an central extension by $\mu_2$ hence we may reduce $\overline{\operatorname{Sp}}_{\psi}(2n)$ to an extension
$$1\to \mu_8 \to\widetilde{\operatorname{Sp}}(2n) \to \operatorname{Sp}(2n) \to 1 .$$

For any subgroup $H \subset \operatorname{Sp}(2n)$, denote the inverse image of $H$ in $\widetilde{\operatorname{Sp}}(2n)$ by $\widetilde{H}$. One of the advantages of working with extensions by $\mu_8$, instead of working with the usual $\mu_2$ extensions, is that: for any standard Levi subgroup $M=  \prod_{i=1}^k\operatorname{GL}(n_i) \times \operatorname{Sp}(2m) \subset \operatorname{Sp(2n)}$, we have an isomorphism (depends on the fixed additive character $\psi$)
$$\tilde{M} \simeq   \prod_{i=1}^k\operatorname{GL}(n_i) \times \widetilde{\operatorname{Sp}}(2m).$$
We will say $\tilde{M}$ is a standard Levi subgroup of $\widetilde{\operatorname{Sp}}(2n)$. A group of the form $\prod_{i}\operatorname{GL}(n_i) \times \widetilde{\operatorname{Sp}}(2m)$ will be called a group of metaplectic type.

Let $H$ be a reductive group  or a group of meteplectic type. For any standard Levi subgroup $L \subset H$,  the normalized parabolic induction (resp. normalized Jacquet module) is denoted by  $i_L^{H}$ (resp. $r_{L}^H$). We view them as maps between the spaces of virtual (genuine) characters over relevant groups.

\subsection{Endoscopy}
\paragraph{} From now on, we fix an positive integer $n$ and set $G=\operatorname{Sp}(2n)$, $\tilde{G}=\widetilde{\operatorname{Sp}}(2n)$.
\paragraph{} Following \cite{li2011transfert} (see also \cite[Chapter 3]{li2021stabilization} for a summary), an elliptic endoscopic datum  of  $\tilde{G}$ is a triple $(G^{!},n^{\prime},n^{\prime\prime})$, where 
$$(n^{\prime},n^{\prime\prime})\in \mathbb{Z}_{\geq 0}^2, n^{\prime}+n^{\prime\prime}=n$$ 
$$G^{!}:=\operatorname{SO}(2n^{\prime}+1)\times \operatorname{SO}(2n^{\prime\prime}+1)$$
and $G^!$ is the endoscopic group attached to the datum $(G^!,n^{\prime},n^{\prime\prime})$.  If there is no ambiguity, an endoscopic group and endoscopic datum will both be denoted as $G^!$. We would like to mention that $(n^{\prime},n^{\prime\prime})$ and $(n^{\prime\prime},n^{\prime})$ will give rise to unequivalent endoscopic data which is a new feature of endoscopy in metaplectic group case.

For $\tilde{M} \simeq   \prod_{i=1}^k\operatorname{GL}(n_i) \times \widetilde{\operatorname{Sp}}(2m)$ a group of metaplectic type, an elliptic endoscopic group of $\tilde{M}$ is defined to be a group of the form $\prod_{i=1}^k\operatorname{GL}(n_i)\times M^!$ where $M^!$ is an elliptic endoscopic group of $\widetilde{\operatorname{Sp}}(2m).$

For $\tilde{G}$, the notion of stable conjugacy is defined in \cite[Section 5.2]{li2011transfert}.  The set of stable regular semisimple conjugacy classes in $\tilde{G}$ is denoted by  $\Sigma_{reg}(\tilde{G})$ and let $\Sigma_{reg}(G^!)$ be the set of stable strongly regular semisimple conjugacy classes in $G^!$. Then a correspondence between semisimple elements 
$$\Psi_{G^!,G}:  \Sigma_{reg}(G^!)\to\Sigma_{reg}(\tilde{G})$$
is defined as in \cite[Section 5.1]{li2011transfert}. This map can be made explicit in terms of eigenvalues. Suppose $\delta=(\delta^{\prime}, \delta^{\prime\prime} )\in G^!$ and let $((a_i^{\prime})^{\pm 1}_{i=1,...,n^{\prime}},1)$ (resp. $((a_i^{\prime\prime})^{\pm 1}_{i=1,...,n^{\prime\prime}},1)$) be the eigenvalues of $\delta^{\prime}$ (resp. $\delta^{\prime\prime}$). Then the eigenvalues of $\Psi_{G^!,G}(\delta)$ are $(a_i^{\prime})^{\pm 1}_{i=1,...,n^{\prime}}, (-a_i^{\prime\prime})^{\pm 1}_{i=1,...,n^{\prime\prime}}$.

Let $\Sigma_{G-reg}(G^!)\subset \Sigma_{reg}(G^{!})$ be the subset consisting of  $\delta$ such that $\Psi_{G^!,G}(\delta)$ is regular and let $\Gamma_{reg}(G)$ be the set of regular semisimple conjugacy classes in $G$. The transfer factor in this setting is a map 
$$\Delta_{G^!,G}(-,-): \Sigma_{G-reg}(G^!) \times \Gamma_{reg}(G) \to \mu_8$$
 defined in \cite[Section 5.3]{li2011transfert}. 
 
 The  transfer factor  satisfies parabolic descent \cite[Section 5.4]{li2011transfert}. Suppose  $\delta\in G^! $ and $\tilde{\gamma} \in \tilde{G}$ belong to Levi subgroups
 $$L=( \prod_{i}\operatorname{GL}(n_i^{\prime}) \times \operatorname{SO}(2m^{\prime}+1) )\times(  \prod_{i}\operatorname{GL}(n_i^{\prime\prime})\times \operatorname{SO}(2m^{\prime\prime}+1))$$
 $$\tilde{M}= \prod_{i=1}^k\operatorname{GL}(n_i) \times \widetilde{\operatorname{Sp}}(2m)$$
 respectively, where $n^{'}_i+n^{''}_i=n_i$. Then we can write 
 $$\delta= (\{\delta_i^{\prime}\}_i,\delta^{\prime})\times(\{\delta^{\prime\prime}_i\}_i,\delta^{\prime\prime})$$
 $$\tilde{\gamma}=(\{\tilde{\gamma}_i\},\tilde{\gamma}^\flat)$$
 where $\delta_i^{\prime}\in \operatorname{GL}(n_i^{\prime})$, $\delta^{\prime}\in \operatorname{SO}(2m^{\prime}+1)$ etc. Note that $L^{\flat}:=\operatorname{SO}(2m^{\prime}+1)  \times \operatorname{SO}(2m^{\prime\prime}+1)$ is an elliptic endoscopic group for $\tilde{G}^{\flat}:=\operatorname{Sp}(2m)$, the parabolic descent for transfer factor means we have
 $$\Delta_{G^!,G}(\delta,\tilde{\gamma})=\Delta_{L^{\flat},\tilde{G}^{\flat}}((\delta^{\prime},\delta^{\prime\prime}),\tilde{\gamma}^{\flat}).$$
 
 Let $\mathcal{T}^{\vee}_{G^!,\tilde{G}}$ be the transfer map for distributions defined as in \cite[Section 3.8]{li2021stabilization}. It follows from the results of  \cite{li2019spectral} that if $\Theta$ is a stable virtual character of $G^!$ then $\mathcal{T}^{\vee}_{G^!,\tilde{G}}(\Theta)$ is a virtual genuine character of $\tilde{G}$. In this case, we have
$$\mathcal{T}^{\vee}_{G^!,\tilde{G}}(\Theta)(\tilde{\gamma})= \sum_{\substack{\delta \in \Sigma_{\text{reg}}(G^!) \\ \Psi_{G^!,\tilde{G}}(\delta)=\gamma }} \Delta_{G^!,\Tilde{G}}(\delta,\Tilde{\gamma})\Theta(\delta)$$
holds for $\forall \gamma\in \Gamma_{reg}(G)$ and $\tilde{\gamma}$ can be any inverse image of $\gamma$.

The formulations above can be extended to the case when $\tilde{G}$ is replace by a group of metaplectic
type by treating $\operatorname{GL}$ factors and the metaplectic factor separately. Suppose $G^!$ is an endoscopic datum of $\tilde{G}$ then by definition $G^!$ is an elliptic endoscopic datum of some standard Levi $\tilde{M}$ of $\tilde{G}$. Then we can extend the definition of $\Psi, \Delta$ to the pair $(G^!,\tilde{G})$ via the inclusion $\tilde{M} \subset \tilde{G}$ and set $$\mathcal{T}^{\vee}_{G^!,\tilde{G}}=i_{\tilde{M}}^{\tilde{G}}\circ \mathcal{T}^{\vee}_{G^!,\tilde{M}}$$
all the formulations above extend to this case as well.
\section{Jacquet module and endoscopy}\label{3}
\paragraph{} In this section, we consider the relation between Jacquet module and endoscopic transfer for metaplectic groups. 
\paragraph{}Let $G^!=\operatorname{SO}(2n^{\prime}+1)\times \operatorname{SO}(2n^{\prime\prime}+1)$ be an elliptic endoscopic datum of $\Tilde{G}$ and let  $M=  \prod_{i=1}^k\operatorname{GL}(n_i) \times \operatorname{Sp}(2m)$ be a standard Levi subgroup  of $G$. We define $\mathcal{E}(\tilde{M},G^!)$ as the set of all sequences $\{(n^{'}_i,n^{''}_i)\}_{i=1,...,k}$ satisfying:
\begin{itemize}
    \item For $\forall 1\leq i\leq k$, $n^{'}_i,n^{''}_i \in \mathbb{Z}_{\geq 0}$ and  $n^{'}_i+n^{''}_i=n_i$
    \item $\sum n_i^{\prime} \leq n^{\prime}$ and $\sum n_i^{\prime \prime}\leq n^{\prime \prime}$.
\end{itemize}

For any $s \in \mathcal{E}({\tilde{M},G^!})$, we set
$$M_s:=(\prod_{i}\operatorname{GL}(n_i^{\prime}) \times \prod_{i}\operatorname{GL}(n_i^{\prime\prime})) \times \operatorname{Sp}(2m)$$
$$M_{s}^! := ( \prod_{i}\operatorname{GL}(n_i^{\prime}) \times \operatorname{SO}(2m^{\prime}+1) )\times(  \prod_{i}\operatorname{GL}(n_i^{\prime\prime})\times \operatorname{SO}(2m^{\prime\prime}+1))$$  
where $m^{\prime}=n^{\prime}-\sum n_i^{\prime}$ and $m^{\prime\prime}=n^{\prime\prime}-\sum n_i^{\prime\prime}$ 
and we view $M_s^!$ as a standard Levi subgroup of $G^{!}$ via inclusions 
$$\prod_{i}\operatorname{GL}(n_i^{\prime}) \times \operatorname{SO}(2m^{\prime}+1) \hookrightarrow \operatorname{SO}(2n^{\prime}+1)$$
$$\prod_{i}\operatorname{GL}(n_i^{\prime\prime})\times \operatorname{SO}(2m^{\prime\prime}+1) \hookrightarrow \operatorname{SO}(2n^{\prime\prime}+1).$$
Here the product respects the order of subscripts and subscripts $i$ with $n_i^{\prime}=0$ (resp. $n_i^{\prime\prime}=0$ ) are omitted. Note that $M_s^!$ is an elliptic endoscopic group of $\tilde{M}_s$ and $M_s$ is a standard Levi subgroup of $M$. Hence we can also view $M_s^!$ as an endoscopic group (not necessarily elliptic) of $\Tilde{M}$.

Let $z_s \in M_s^!$ be the central element whose projections to $\operatorname{SO}$ factors, $\operatorname{GL}(n_i^{\prime})$ are $1$ and projections to $\operatorname{GL}(n_i^{\prime\prime})$ are $-1$. Multiplication by $z_s$ induces a map on distributions over $M_s^!$ which we denote by $[z_s]$. Since $z_s$ has order $2$, $[z_s]$ is an involution.
With these preparations, we can now state the relation between $\mathcal{T}^{\vee}$ and Jacquet module, which is an analogue of \cite[Proposition 5.6]{Hiraga}.

\begin{theorem}\label{jacquet}
    For any stable virtual character $\Theta$, we have 
    $$r_{\Tilde{M}}^{\Tilde{G}}\circ \mathcal{T}^{\vee}_{G^!,\Tilde{G}}(\Theta)=\sum_{s\in \mathcal{E}(M,G^!)} \mathcal{T}^{\vee}_{M_s^!,\tilde{M}} \circ [z_s] \circ  r^{G^!}_{M_s^!}(\Theta)$$
\end{theorem}

Here the right hand side make sense because $r^{G^!}_{M_s^!}$ maps stable virtual characters to stable virtual characters (see eg. \cite[Lemma 3.3]{Hiraga}).
\begin{lemma}\label{bijection}
    For $\gamma \in M_{G-{\text{reg}}}$, the following sets are in bijection:
    \begin{itemize}
        \item $\{\delta \in \Sigma_{\text{reg}}(G^!)| \Psi_{G^!,G}(\delta)=\gamma \}$
        \item $\{(s,\delta_s )| s\in \mathcal{E}(M,G^!) \text{, } \delta_{s}\in \Sigma_{\text{reg}}(M_s^!), \Psi_{M^!_s,M}(\delta_s \cdot z_s)= \gamma\}$
    \end{itemize}
\end{lemma}

\begin{proof}
    For $\delta$ in the first set, let $\delta^{\prime}$ (resp. $\delta^{\prime\prime}$) be its projection to $\operatorname{SO}(2n^{\prime}+1)$ (resp. $\operatorname{SO}(2n^{\prime\prime}+1)$ ). Let $\gamma_i$ be the projection of $\gamma$ to $\operatorname{GL}(n_i)$. For each $1 \leq i\leq k$, let $2n_i^{\prime}$ be the number of eigenvalues of $\delta^{\prime}$ which is also an eigenvalue of $\gamma_i$.  Similarly, we define $2n_i^{\prime\prime}$ to be the number of eigenvalues of $\delta^{\prime\prime}$ which is also an eigenvalue of $-\gamma_i$. Then since $\gamma$ is regular, $n_i^{\prime}, n^{\prime\prime}_{i}$ are well-defined integers and we have $n_i^{\prime}+n_i^{\prime\prime}=n_i$. Thus we obtain an element  in $ \mathcal{E}(M,G^!)$  given by $s=\{(n_i^{\prime},n_i^{\prime\prime})\}_{i=1,...,k}$ and $\delta$ can be conjugated to some element $\delta_s\in M_s^!$. By \cite[Proposition 3.4.9]{li2021stabilization}, $(s,\delta_s)$ belongs to the second set.
\end{proof}

\begin{proof}[Proof of Theorem \ref{jacquet}]
    Let $A_M$ be the center of $M$ and let $\Delta_G$ (resp. $\Delta_M$) be the set of simple roots of $G$ (resp. $M$) for our fixed choice of Borel pair. Define $$A_M^-:=\{ t\in A_M \ |\  \forall \alpha\in \Delta_G-\Delta_M \text{ 
 , }|\alpha(t)|< 1\}.$$ Then Casselman's character formula, where a proof for covering group setting can be found in \cite{luo2017howe}, implies for any $\gamma\in M_{G-\text{reg}}$, $a\in A_M^{-}$  and $n\gg 0$  we have 
 
    \begin{align*}
        r_{\Tilde{M}}^{\Tilde{G}}\circ \mathcal{T}^{\vee}_{G^!,\tilde{G}}(\Theta)(a^n\Tilde{\gamma})&= \mathcal{T}^{\vee}_{G^!,\tilde{G}}(\Theta)(a^n\Tilde{\gamma})\\
        &= \sum_{\substack{\delta \in \Sigma_{\text{reg}}(G^!) \\ \Psi_{G^!,G}(\delta)=a^n\gamma }}\Delta_{G^!,\Tilde{G}}(\delta,a^n\Tilde{\gamma})\Theta(\delta)
    \end{align*}
     where $\Tilde{\gamma}$ is an inverse image of $\gamma$ in $\tilde{G}$.
    On the other hand, for each $s \in \mathcal{E}(M,G^!)$, we have 
    \begin{align*}
        \mathcal{T}^{\vee}_{M_s^!,\Tilde{M}} \circ [z_s] \circ  r^{G^!}_{M_s^!}(\Theta)(a^n\Tilde{\gamma})&=\sum_{\substack{\delta_{s}\in \Sigma_{\text{reg}}(M_s^!)\\ \Psi_{M_s^!,M}(\delta_s )= a^n\gamma}} \Delta_{M_s^!,\Tilde{M}}(\delta_s,a^n\Tilde{\gamma})[z_s]\circ r^{G^!}_{M_s^!}(\Theta)(\delta_s) \\
        &= \sum_{\substack{\delta_{s}\in \Sigma_{\text{reg}}(M_s^!)\\ \Psi_{M_s^!,M}(\delta_s\cdot z_s)= a^n\gamma}} \Delta_{M_s^!,\Tilde{M}}(\delta_s \cdot z_s,a^n\Tilde{\gamma})\Theta(\delta_s) 
    \end{align*}
    where the second equality is obtained via replacing $\delta_s$ by $\delta_s \cdot z_s$. Note that multiplication by $z_s$ only affects $\operatorname{GL}$ factors, hence parabolic descent for transfer factors  implies 
    $$ \Delta_{M_s^!,\Tilde{M}}(\delta_s \cdot z_s,a^n\Tilde{\gamma})= \Delta_{G^!,\Tilde{G}}(\delta_s,a^n\Tilde{\gamma}).$$ 
    Take summation over $s\in \mathcal{E}(M,G^!)$, then by Lemma \ref{bijection} we have 
    $$r_{\Tilde{M}}^{\Tilde{G}}\circ \mathcal{T}^{\vee}_{G^!,\Tilde{G}}(\Theta)(a^n\tilde{\gamma})=\sum_{s\in \mathcal{E}(M,G^!)} \mathcal{T}^{\vee}_{M_s^!,\tilde{M}} \circ [z_s] \circ  r^{G^!}_{M_s^!}(\Theta)(a^n\tilde{\gamma})$$
    holds for any $\gamma\in M_{G-\text{reg}}$, $a\in A_M^{-}$  and $n\gg 0$. Hence the virtual characters on two sides are equal.
\end{proof}

\paragraph{}We end this section by a corollary of Theorem \ref{jacquet} about partial Jacquet modules. It will not be used in the following,  we include it here mainly for future reference.

  Let $M=\operatorname{GL}(d)\times \operatorname{Sp}(2m)$, where $m=n-d$,  be a maximal Levi subgroup of $G$. For a virtual  genuine character $\Theta$ of $\tilde{G}$, we can write
$$r^{\tilde{G}}_{\tilde{M}}(\Theta)=\sum_{i\in I}a_i\Theta_{\pi_i}$$
for some indexing set $I$, some irreducible  genuine representations  $\pi_i$  of $\tilde{M}$ and complex numbers $a_i$.
Each $\pi_i$ takes the form $\rho_i \boxtimes \pi^{\prime}_i$ where $\rho$ is an irreducible representation of $\operatorname{GL}(d)$ and $\pi_i^{\prime}$ is an irreducible genuine representation of $\widetilde{\operatorname{Sp}}(2m)$.

Let $\rho$ be a unitary irreducible supercuspidal representation of $\operatorname{GL}(d)$. For a real number $x$, let $|\cdot|^x$ be the character of $\operatorname{GL}(d)$ given by $g\mapsto |\operatorname{det}(g)|^x$. Then the partial Jacquet module with respect to $\rho,x$ is defined as 
$$r_{\rho,x}(\Theta):= \sum_{\substack{i \in I \\ \rho_{i} \simeq \rho|\cdot|^x} }a_i\Theta_{\pi_i^{\prime}}$$
which is a virtual genuine character of $\widetilde{\operatorname{Sp}}(2m)$.

We would like to apply $r_{\rho,x}$ to $\mathcal{T}^{\vee}_{G^!,\Tilde{G}}(\Theta)$ where $\Theta$ is a stable virtual character of $G^!$. In the current setting, we have 
$$\mathcal{E}(M,G^!)=\{(d^{\prime},d^{\prime\prime})\ | \ d^{\prime},d^{\prime\prime} \in \mathbb{Z}_{\geq0}\text{ , }d^{\prime}+d^{\prime\prime}=d \text{ , } d^{\prime}\leq n^{\prime}, d^{\prime\prime} \leq n^{\prime\prime}\}$$ and Levi subgroups appear on right hand side of Theorem \ref{jacquet} takes the form 
$$L=(\operatorname{GL}(d^{\prime})\times \operatorname{SO}(2m^{\prime}+1))\times(\operatorname{GL}(d^{\prime\prime})\times \operatorname{SO}(2m^{\prime\prime}+1)).$$ 
The endoscopic transfer $\mathcal{T}^{\vee}_{L,\tilde{M}}$ map then factors as 
$$L\xrightarrow[endoscopy]{elliptic}\operatorname{GL}(d^{\prime})\times\operatorname{GL}(d^{\prime\prime})\times \widetilde{\operatorname{Sp}}(2m) \xrightarrow[induction]{parabolic} \tilde{M}.$$
Because $\rho$ is a supercuspidal representation, $\mathcal{T}^{\vee}_{M_s^!,\tilde{M}} \circ [z_s] \circ  r^{G^!}_{M_s^!}(\Theta)$ has no components isomorphic to $\rho|\cdot|^x\boxtimes \pi^{\prime}$ unless $s=(d,0) \text{ or } (0,d)$. 

Now let $\phi$ be a discrete L-parameter of $\operatorname{SO}(2n+1)$. By definition, $\phi$ is a $2n$-dimensional symplectic representation of  the Weil-Deligne group and it can be written as a multiplicity free direct sum 
$$\phi=\oplus \phi_i \boxtimes S_{a_i-1}$$
where each $\phi_i$ is an irreducible unitary representation of Weil group and $S_{a_i-1}$ is the $a$-dimensional irreducible representation of $\operatorname{SL}(2,\mathbb{C})$. By local Langlands correspondence for $\operatorname{GL}$, each $\phi_i$ corresponds to an irreducible supercuspidal representation $\rho_i$. The set of Jordan blocks of $\phi$ is defined as 
$$Jord(\phi):=\{(\rho_i,a_i)| \phi_i\boxtimes S_{a-1}\subset \phi\}.$$
Suppose $\phi$ factors through $G^!$, then we can write $\phi= (\phi^{\prime},\phi^{\prime\prime})$. We take $(\rho,a)\in Jord(\phi)$ and set $x=\frac{a-1}{2}$.  Let   $\Theta_{\phi^{\prime}}$ (resp. $\Theta_{\phi^{\prime}}$)  be the stable virtual character attached to $\phi^{\prime}$(resp. $\phi^{\prime\prime}$). If $r_{\rho,x}(\Theta_{\phi^{\prime}})\neq 0$ (resp. $r_{\rho,x}(\Theta_{\phi^{\prime\prime}})\neq 0$) then $(\rho,a)$ is contained in $Jord(\phi^{\prime})$ (resp. $Jord(\phi^{\prime\prime})$), see eg. \cite[Lemma 7.2]{xu2017cuspidal}. Since $\phi$ is discrete, at most one case can happen.

In summary, we can obtain the following:
\begin{corollary}
    Let $\phi=(\phi^{\prime},\phi^{\prime\prime})$ be a discrete L-parameter of $SO(2n+1)$ that factors through $G^!$ and let $\Theta_{\phi}$ be the stable virtual character of $G^!$ attached to $\phi$. We take a standard Levi $M^!\subset G^!$ defined as 
    \begin{align*}
        M^!:=
        \begin{cases}
           M_{(d,0)}^!& \text{ if } (\rho,a) \in Jord(\phi^{\prime}) \\
            M_{(0,d)}^! & \text{ if } (\rho,a) \in Jord(\phi^{\prime\prime}).
        \end{cases}
    \end{align*}
    Write $\tilde{G}_{-}=\widetilde{Sp}(2m)$ and let $G_{-}^!$ be the $\operatorname{SO}$ part of $M^!$ then we have
    $$\alpha r_{\rho,x} \circ \mathcal{T}^{\vee}_{G,\tilde{G}}(\Theta_{\phi})=\mathcal{T}^{\vee}_{G^!_{-},\tilde{G}_{-}} \circ r_{\rho,x}(\Theta_\phi),$$
    where $\alpha$ is a sign defined as  
    \begin{align*}
        \alpha=
        \begin{cases}
            1& \text{if  } (\rho,a) \in Jord(\phi^{\prime}) \\
            \omega_{\rho}(-1) & \text{if  } (\rho,a) \in Jord(\phi^{\prime\prime})\\
        \end{cases}
    \end{align*}
    with  $ \omega_{\rho}$ stands for the central character of $ \rho$.
\end{corollary}
\section{Compatibility}

\paragraph{}Now we start to prove  Theorem \ref{main}. First we consider the case where $G^!$ is an elliptic endoscopic group of $G$.

For a split reductive group $H$ with a fixed Borel pair, let $\mathcal{L}(H)$ be the set of standard Levi subgroups of $H$. Recall from \cite{aubert1995dualite} that the  Aubert-Zelevinski involution is defined as 
$$D_{H}:=\sum_{L\in \mathcal{L}(H)}(-1)^{r(L)}i^{H}_{L}\circ r^{H}_L$$
where $r(-)$ stands for semisimple rank. For $\tilde{G}$, its Aubert-Zelevinski involution is defined similarly as
$${D}_{\Tilde{G}}: =\sum_{M\in \mathcal{L}(G)}(-1)^{r(M)}i^{\Tilde{G}}_{\Tilde{M}}\circ r_{\Tilde{M}}^{\Tilde{G}}.$$

Let $G^!=\operatorname{SO}(2n^{\prime}+1)\times \operatorname{SO}(2n^{\prime\prime}+1)$ be  an elliptic endoscopic group of $\tilde{G}$.  For $L \in \mathcal{L}(G^!)$, define 
$$\mathcal{M}(L):=\{(M,s)| M\in \mathcal{L}(G)\text{, } s\in \mathcal{E}(M,G^!), L=M_s^!\}$$
By \cite[Section 3.8]{li2021stabilization}, we know 

$$\mathcal{T}^{\vee}_{G^!,\tilde{G}} \circ i_{M_s^!}^{G^!} \circ [z_s] = i_{\Tilde{M}}^{\Tilde{G}}\circ \mathcal{T}^{\vee}_{M_s^!,\tilde{M}} $$
holds for any $M\in \mathcal{L}(G)$, $s\in \mathcal{E}(M,G^!)$. Then we can compute
\begin{align*}
    {D}_{\Tilde{G}}\circ \mathcal{T}^{\vee}_{G^!,\tilde{G}}&=\sum_{M\in \mathcal{L}(G)}(-1)^{r(M)}i^{\Tilde{G}}_{\Tilde{M}}\circ r_{\Tilde{M}}^{\Tilde{G}} \circ \mathcal{T}^{\vee}_{G^!,\tilde{G}}\\
    &=\sum_{M\in \mathcal{L}(G)}(-1)^{r(M)} i^{\Tilde{G}}_{\Tilde{M}}(\sum_{s\in \mathcal{E}(M,G^!)} \mathcal{T}^{\vee}_{M_s^!,\tilde{M}}\circ [z_s] \circ  r^{G^!}_{M_s^!})\\
    &= \sum_{M\in \mathcal{L}(G)}(-1)^{r(M)} \sum_{s\in \mathcal{E}(M,G^!)} \mathcal{T}^{\vee}_{G^!,\tilde{G}}\circ i^{G^!}_{M_s^!} \circ  r^{G^!}_{M_s^!} \\
    &= \mathcal{T}^{\vee}_{G^!,\tilde{G}}(\sum_{L\in \mathcal{L}(G^!)}(\sum_{(M,s)\in \mathcal{M}(L)}(-1)^{r(M)} i^{G^!}_{L}\circ r^{G^!}_L))
\end{align*}
Comparing with the definition of $D_{G^!}$, we see that it is enough to show $$\sum_{(M,s)\in \mathcal{M}(L)}(-1)^{r(M)}=(-1)^{r(L)}$$ 
holds for all $L\in \mathcal{L}(G^!)$. Suppose $L$ takes the form  
$$L=( \prod_{i=1}^{k^{\prime}}\operatorname{GL}(n_i^{\prime}) \times \operatorname{SO}(2m^{\prime}+1) )\times(  \prod_{i=1}^{k^{\prime \prime}}\operatorname{GL}(n_i^{\prime\prime})\times \operatorname{SO}(2m^{\prime\prime}+1)).$$
Here, if $k^{\prime}$ or $k^{\prime\prime}$ is $0$ then the corresponding product is understood as taken over  empty set, similar conventions apply to later discussion. Write  $I^{\prime}=(n_1^{\prime},...,n_{k^{\prime}}^{\prime} )$ and define $I^{\prime\prime}$ similarly. 

\begin{lemma}
    The set $\mathcal{M}(L)$ is in bijection with the set of triple $(k,\Bar{I}^{\prime},\Bar{I}^{\prime\prime})$ satisfying:
\begin{itemize}
    \item $k\in \mathbb{Z}$ and $\operatorname{max}(k^{\prime},k^{\prime\prime})\leq k$
    \item $\Bar{I}^{\prime}=(\Bar{n}^{\prime}_1,....,\Bar{n}^{\prime}_k)$ and there exists a subsequence $1\leq i_1<...<i_{k^{\prime}}\leq k$ such that 
    \begin{align*}
        \Bar{n}^{\prime}_i=
        \begin{cases}
            n_j^{\prime} &\text{ if } i=i_j \\
            0 &\text{ else }
        \end{cases}.
    \end{align*}
    Similar condition holds for $\Bar{I}^{\prime\prime}$.
    \item For $\forall 1\leq i \leq k$, at least one of $\Bar{n}_i^{\prime},\Bar{n}_i^{\prime\prime}$ is non-zero
\end{itemize}
\end{lemma} 
\begin{proof}
    Given a triple $(k,\Bar{I}^{\prime},\Bar{I}^{\prime\prime})$ as above, the corresponding Levi is 
$$M=\prod_{i=1}^k\operatorname{GL}(\Bar{n}_i^{\prime}+\Bar{n}_i^{\prime\prime}) \times \operatorname{Sp}(2m)$$
and the sequence $\{(\Bar{n}_i^{\prime},\Bar{n}_i^{\prime\prime})\}_{i=1,...,k}$ gives an element $s\in \mathcal{E}(M,G^!)$ with $M^!_s=L$.

Conversely, if $M_s^!=L$ where $s=\{(\bar{n}^{'}_i,\bar{n}^{''}_i)\}_{i=1,...,k}$ then, by the definition of $M_s^!$, there is a subsequence of $\{\bar{n}_i^{\prime}\}_i$ (resp. $\{\bar{n}^{\prime\prime}_i\}_i$) which is equal to $I^{\prime}$ (resp. $I^{\prime\prime}$) and  entries of $\{\bar{n}_i^{\prime}\}_i$ (resp. $\{\bar{n}_i^{\prime\prime}\}_i$) don't belong to this subsequence equal to $0$. Hence $(k,\{(\bar{n}^{'}_i\},\{\bar{n}^{''}_i\})$ is a triple as above. Clearly the above two maps are inverse to each other.
\end{proof}
 Note that the latter set of triples depends only on the pair $(k^{\prime},k^{\prime\prime})$, thus we may denote it as $\mathcal{M}(k^{\prime},k^{\prime \prime})$. Note that we have $r(L)=r(G)-(k^{\prime}+k^{\prime\prime})$ and if $M$ is the standard Levi attached to a triple $(k,\Bar{I}^{\prime},\Bar{I}^{\prime\prime}) \in \mathcal{M}(k^{\prime},k^{\prime \prime})$ then we also have $r(M)=r(G)-k$. Therefore if we set
$$f(k^{\prime},k^{\prime\prime}):=\sum_{(k,\Bar{I}^{\prime},\Bar{I}^{\prime\prime})\in \mathcal{M}(k^{\prime},k^{\prime \prime})} (-1)^{k}$$
then we have 
$$\sum_{(M,s)\in \mathcal{M}(L)}(-1)^{r(M)}=(-1)^{r(G)}f(k^{\prime},k^{\prime\prime}).$$
Hence we are reduced to proving the following lemma:
\begin{lemma}
    $f(k^{\prime},k^{\prime\prime})=(-1)^{k^{\prime}+k^{\prime\prime}}$
\end{lemma}
\begin{proof}
    We can divide $\mathcal{M}(k^{\prime},k^{\prime \prime})$ into three subsets according to the following conditions:
    \begin{itemize}
        \item The first entry of $\Bar{I}^{\prime}$ is non-zero while the first entry of $\Bar{I}^{\prime\prime}$ is zero.
        \item The first entry of $\Bar{I}^{\prime\prime}$ is non-zero while the first entry of $\Bar{I}^{\prime}$ is zero.
        \item The first entries of $\Bar{I}^{\prime}, \Bar{I}^{\prime\prime}$ are both non-zero.
    \end{itemize}
    Then by deleting the first entries of $\Bar{I}^{\prime}, \Bar{I}^{\prime\prime}$ simultaneously,  the above subsets are in bijection with $\mathcal{M}(k^{\prime}-1,k^{\prime\prime}), \mathcal{M}(k^{\prime},k^{\prime \prime}-1), \mathcal{M}(k^{\prime}-1,k^{\prime \prime}-1)$ respectively. Thus we have
    $$f(k^{\prime},k^{\prime\prime})=-f(k^{\prime}-1,k^{\prime\prime})-f(k^{\prime},k^{\prime\prime}-1)-f(k^{\prime}-1,k^{\prime\prime}-1)$$
    and the lemma follows from induction on $k^{\prime},k^{\prime\prime}$.
\end{proof}

Now we move to general case. Let $G^!$ be an (not necessarily elliptic) endoscopic group of $\tilde{G}$. Then there exists a standard Levi subgroup $\tilde{M}$ such that $G^!$ is an elliptic endoscopic group of $\tilde{M}$. By \cite[Theorem 1.7]{aubert1995dualite} we have
$$i_{\tilde{M}}^{\tilde{G}}\circ D_{\tilde{M}}=D_{\tilde{G}} \circ i_{\tilde{M}}^{\tilde{G}}.$$
On the other hand, since $\mathcal{T}^{\vee}_{G^!,\tilde{M}}$ is identity on $\operatorname{GL}$ factors, the above ellip
$${D}_{\tilde{M}}\circ \mathcal{T}^{\vee}_{G^!,\tilde{M}}= \mathcal{T}^{\vee}_{G^!,\tilde{M}}\circ D_{G^!}$$
Therefore we have 
\begin{align*}
    {D}_{\tilde{G}}\circ \mathcal{T}^{\vee}_{G^!,\tilde{G}} &=D_{\tilde{G}}\circ  i^{\tilde{G}}_{\tilde{M}} \circ \mathcal{T}^{\vee}_{G^!,\tilde{M}}\\
    &= i^{\tilde{G}}_{\tilde{M}} \circ \mathcal{T}^{\vee}_{G^!,\tilde{M}}\circ D_{G^!}\\
    &= \mathcal{T}^{\vee}_{G^!,\tilde{G}}\circ D_{G^!}
\end{align*}
which finishs the proof of Theorem \ref{main}.

 \renewcommand{\bibname}{References}
	\bibliographystyle{alpha}
	\bibliography{MpAZ}

 \vspace{1em}
\begin{flushleft} \small
	F. Chen: Yau Mathematical Sciences Center, Tsinghua University, Haidian District, Beijing 100084, China
 \\
	E-mail address: \href{mailto:fchen@tsinghua.edu.cn}{\texttt{fchen@tsinghua.edu.cn}}
\end{flushleft}
\end{document}